\documentclass[reqno,10pt]{article}
\font\smallit=cmti10

\usepackage{amssymb,latexsym,amsmath,epsfig,amsthm} 

\begin{document}

\newtheorem{theorem}{Theorem}         
\newtheorem{proposition}{Proposition}
\newtheorem{cor}[proposition]{Corollary}
\newtheorem{lemma}[proposition]{Lemma}
\newtheorem{remark}[proposition]{Remark}
\newtheorem{definition}[proposition]{Definition}
\newtheorem{example}[proposition]{Example}

\numberwithin{equation}{section}

\renewcommand{\AA}{\mathcal A}
\newcommand{\Z}{\mathbb Z}
\newcommand{\N}{\mathbb N}
\newcommand{\R}{\mathbb R}
\newcommand{\C}{\mathbb C}
\newcommand{\Q}{\mathbb Q}
\newcommand{\RRR}{\mathfrak R}
\newcommand{\GG}{\mathcal G}
\newcommand{\A}{\mathfrak A}
\newcommand{\VV}{\mathcal V}
\newcommand{\EE}{\mathcal E}
\newcommand{\G}{\mathbb G}
\newcommand{\F}{\mathbb F}
\renewcommand{\H}{\mathbb H}
\newcommand{\I}{\mathbb I}
\newcommand{\DD}{\mathcal D}
\newcommand{\CC}{\mathcal C}
\newcommand{\LL}{\mathcal L}
\newcommand{\RR}{\mathcal R}
\newcommand{\BB}{\mathcal B}
\newcommand{\M}{\mathbb M}
\newcommand{\B}{\mathbb B}
\newcommand{\K}{\mathbb K}
\newcommand{\II}{\mathfrak I}
\newcommand{\oF}{\overline{F}}
\newcommand{\KK}{\mathcal K}
\newcommand{\OO}{\mathcal O}
\newcommand{\QQ}{\mathcal Q}
\newcommand{\MM}{\mathcal M}
\newcommand{\wMM}{\widetilde{\RRR}}
\newcommand{\NN}{\mathcal N}
\newcommand{\TT}{\mathcal T}
\newcommand{\FF}{\mathcal F}
\newcommand{\bb}{\widetilde{\beta}}
\newcommand{\we}{\widetilde{e}}
\newcommand{\wf}{\widetilde{f}}
\newcommand{\Prim}{\operatorname{Prim}}
\newcommand{\boplus}{\mbox{\small $\displaystyle \bigoplus$}}
\newcommand{\oee}{\overline{e}}
\renewcommand{\Re}{\operatorname{Re}}
\newcommand{\veps}{\varepsilon}
\newcommand{\tr}{\operatorname{Tr}}
\newcommand{\wFF}{\widetilde{\FF}}
\newcommand{\comment}[1]{}

\begin{center}
\uppercase{\bf Gap distribution of Farey fractions under \\ some divisibility constraints}
\vskip 20pt
{\bf Florin P. Boca\footnote{Member of the Institute of Mathematics ``Simion Stoilow" of the Romanian
Academy, 21 Calea Grivi\c tei, 010702 Bucharest, Romania}}\\
{\smallit Department of Mathematics, University of Illinois, Urbana, IL 61801, USA}\\
{\tt fboca@illinois.edu}
\vskip 10pt
{\bf Byron Heersink }\\
{\smallit Department of Mathematics, University of Illinois, Urbana, IL 61801, USA}\\
{\tt heersin2@illinois.edu}
\vskip 10pt
{\bf Paul Spiegelhalter}\\
{\smallit Department of Mathematics, University of Illinois, Urbana, IL 61801, USA}\\
{\tt spiegel3@illinois.edu}
\end{center}
\vskip 30pt

\centerline{\bf Abstract}

\noindent
For a given positive integer $\ell$ we show the existence of the limiting gap distribution measure for the sets of
Farey fractions $\frac{a}{q}$ of order $Q$ with $\ell\nmid a$, and respectively with $(q,\ell)=1$,
as $Q\rightarrow\infty$.

\section{Introduction}
The set $\FF_Q$ of Farey fractions of order $Q$ consists of those rational numbers $\frac{a}{q} \in (0,1]$ with $(a,q)=1$ and $q\leqslant Q$.
The spacing statistics of the increasing sequence $(\FF_Q)$ of finite subsets of $(0,1]$
have been investigated by several authors \cite{Hall,ABCZ,BZ}. Recently
Badziahin and Haynes considered a problem related to the distribution of gaps in the subset $\FF_{Q,d}$ of $\FF_Q$ of those fractions
$\frac{a}{q}$ with $(q,d)=1$, where $d$ is a fixed positive integer and $Q\rightarrow \infty$. They proved \cite{BH}
that, for each $k\in\N$, the number $N_{Q,d}(k)$ of pairs $\big( \frac{a}{q},\frac{a'}{q'}\big)$ of consecutive elements in $\FF_{Q,d}$
with $a^\prime q-aq^\prime =k$ satisfies the asymptotic formula
\begin{equation}\label{1.1}
N_{Q,d}(k)=c(d,k) Q^2 +O_{d,k}(Q\log Q) \qquad (Q\rightarrow\infty),
\end{equation}
for some positive constant $c(d,k)$ that can be expressed using the measure of certain cylinders associated with the
area-preserving transformation introduced by Cobeli, Zaharescu, and the first author in \cite{BCZ}. The pair correlation function of $(\FF_{Q,d})$ was studied and shown to exist
by Xiong and Zaharescu \cite{XZ}, even in the more general situation where $d=d_Q$ is no longer constant but increases according to the rules $d_{Q_1} \mid d_{Q_2}$ as $Q_1<Q_2$ and $d_Q \ll Q^{\log \log Q /4}$.

This paper is concerned with the gap distribution of the sequence of sets $(\FF_{Q,d})$, and respectively of
$(\widetilde{\FF}_{Q,\ell})$, the sequence of sets $\wFF_{Q,\ell}$ of Farey fractions $\gamma=\frac{a}{q}\in\FF_Q$ with
$\ell\nmid a$. Our peculiar interest in $\wFF_{Q,\ell}$ arises from the problem studied in \cite{BG}, concerning the distribution
of the free path associated to the linear flow through $(0,0)$ in $\R^2$ in the small scatterer limit, in the case of circular scatterers of radius $\varepsilon>0$ placed at the points $(m,n)\in \Z^2$ with $\ell \nmid (m-n)$.
When $\ell=3$ this corresponds, after suitable normalization, to the situation of scatterers distributed at the vertices of a honeycomb tessellation, and the linear flow passing through the center of one of the hexagons.
When $\ell=2$ the scatterers are placed at the vertices of a square lattice and the linear flow passes through the center of one the squares. Arithmetic properties of the number $\ell$ are shown to be explicitly reflected by the gap distribution of the elements of $(\wFF_{Q,\ell})$.
The symmetry $x\mapsto 1-x$ shows that for the purpose of studying the gap distribution of these fractions on $[0,1]$ one can replace the condition $\ell \nmid (m-n)$ by
the more esthetic one $\ell \nmid n$.

The \emph{gap distribution} (or \emph{nearest neighbor distribution}) of a numerical sequence, or more generally of a sequence of finite subsets of $[0,1)$,
measures the distribution of lengths of gaps between the elements of the sequence. Let $A=\{ x_0 \leqslant x_1 \leqslant \ldots \leqslant x_N\}$
be a finite list of numbers in $[0,1)$, scaled to $\tilde{x}_j =\frac{Nx_j}{x_N-x_0}$ with mean spacing $\frac{\tilde{x}_N -\tilde{x}_0}{N} =1$.
The \emph{gap distribution measure} of $A$ is the finitely supported probability measure on $[0,\infty)$ defined by
\[
\nu_A (-\infty,\xi]=\nu_A [0,\xi] :=\frac{1}{N} \# \big\{ j\in [1,N] : \tilde{x}_j -\tilde{x}_{j-1} \leqslant \xi \big\} ,\qquad \xi \geqslant 0.
\]
If it exists, the weak limit $\nu=\nu_{\mathcal A}$ of the sequence $(\nu_{A_n})$ of probability measures associated with an increasing sequence ${\mathcal A}=(A_n)$ of finite lists of numbers in $[0,1)$,
is called the \emph{limiting gap measure} of ${\mathcal A}$.

It is elementary (see, e.g., Lemma \ref{L1} below) that
\begin{equation}\label{1.2}
\# \wFF_{Q,\ell} =\widetilde{K}_\ell Q^2 +O_\ell (Q\log Q) ,\qquad \# \FF_{Q,d} =K_d Q^2 +O_d (Q\log Q) ,
\end{equation}
where
\begin{equation*}
\widetilde{K}_\ell = \frac{1}{2\zeta(2)}-\frac{C(\ell)}{2\ell},\quad K_d =\frac{C(d)}{2},\quad
\mbox{\rm with}  \quad  C(\ell)=\frac{1}{\zeta(2)}  \prod\limits_{\substack{p\in {\mathcal P} \\ p\mid \ell}}
\bigg( 1+\frac{1}{p}\bigg)^{-1} .
\end{equation*}

We prove the following result:

\begin{theorem}\label{T1}
Given positive integers $\ell$ and $d$, the limiting gap measures $\widetilde{\nu}_\ell$ of $(\wFF_{Q,\ell})$, and respectively $\nu_d$ of $(\FF_{Q,d})$, exist.
Their densities are continuous on $[0,\infty)$ and
real analytic on each component of $(0,\infty) \setminus \N \widetilde{K}_\ell$, and respectively of $(0,\infty) \setminus \N K_d$.
\end{theorem}

The existence of $\widetilde{\nu}_\ell$ is proved in Section 2 and the limiting gap distribution is explicitly computed
in \eqref{2.9} using
tools from \cite{BCZ}, \cite{BZ1} and \cite{BG}. The result on $\nu_d$ is proved in Section 4. When $d$ is a prime power, an
explicit computation can be done as for $\tilde{\nu}_\ell$. In general the repartition function of $\nu_d$ depends on the measure of some cylinders
associated with the transformation $T$ from \eqref{2.7}, and on the length of strings of consecutive elements
in $\FF_Q$ with at least one denominator relatively prime with $d$.

The upper bound $4d^3$ for $L(d)=\min\{ L: \forall i,\forall Q,\exists j\in [0,L], (q_{i+j},d)=1\}$ was found in \cite{BH},
where $q_i,\ldots,q_{i+L}$ denote the denominators of a string $\gamma_i <\cdots < \gamma_{i+L}$ of consecutive elements in $\FF_Q$.
Although we expect this bound to be considerably smaller, we could only improve it in a limited number of situations.
In Section 3 we lower it to $4\omega (d)^3$ for integers $d$ with the
property that the smallest prime divisor of $d$ is $\geqslant \omega (d)$, where $\omega (d)$ denotes as usual the number of distinct prime
factors of $d$. The bound $L(d)=1$ is trivial when $d$ is a prime power. Employing properties of the transformation $T^2$
we show that $L(d)\leqslant 5$ when $d$ is the product of two prime powers, which is sharp. Finding better bounds
on $L(d)$ when $\omega (d)\geqslant 3$ appears to be an interesting problem in combinatorial number theory.

\section{The gap distribution of $\wFF_{Q,\ell}$}
Let $\mathcal{F}_Q^{(\ell)}=\FF_Q \setminus \wFF_{Q,\ell}$ denote the set of Farey fractions $\gamma=\frac{a}{q}\in \FF_Q$ with $\ell \mid a$, and let
$N_Q^{(\ell)}$ denote the cardinality of $\FF_Q^{(\ell)}$. Consider also:
\[
\begin{split}
\GG_Q (\xi) & :=\bigg\{ (\gamma,\gamma^\prime): \mbox{\rm $\gamma$, $\gamma^\prime$ consecutive in $\FF_Q$}, 0<\gamma^\prime -\gamma \leqslant \frac{\xi}{Q^2} \bigg\},
\\
\GG^{(\ell)}_Q (\xi) & := \bigg\{ (\gamma,\gamma^\prime): \mbox{\rm $\gamma$, $\gamma^\prime$ consecutive in $\wFF_{Q,\ell}$},
0<\gamma^\prime -\gamma \leqslant \frac{\xi}{Q^2} \bigg\} ,
\\  N_Q (\xi): & =\# \GG_Q (\xi),
\quad N^{(\ell)}_Q (\xi):= \# \GG^{(\ell)}_Q (\xi) .
\end{split}
\]

\begin{lemma}\label{L1}
$\displaystyle \  N_Q^{(\ell)} = \frac{C(\ell)}{2\ell} Q^2 +O_\ell (Q\log Q)$ as $Q\rightarrow\infty$.
\end{lemma}

\begin{proof}
It is clear that
$$ N_Q^{(\ell)} = \# \FF_Q^{(\ell)}
= \sum_{\substack{q=1\\(\ell,q)=1}}^Q \sum_{\substack{a=1 \\ (a,q)=1\\ \ell \mid a}}^q 1 .
$$
Letting $k=\frac{a}{\ell}$ and noting that whenever $(\ell,q)=1$
we have $(k\ell,q)=1$ if and only $(k,q)=1$, the sum above becomes
$$
\sum_{\substack{q=1\\(\ell,q)=1}}^Q \sum_{\substack{k=1\\(k,q)=1}}^{[q/\ell]} 1.
$$
Standard M\" obius summation, cf. \eqref{A1} and \eqref{A2}, and $\sum_{q=1}^Q\sigma_0(q)=O(Q\log Q)$, where $\sigma_0(q)=\sum_{d\mid q} 1$, yield
\[
\sum_{\substack{q=1\\(\ell,q)=1}}^Q \sum_{\substack{k=1\\(k,q)=1}}^{[q/\ell]} 1
= \sum_{\substack{q=1\\(\ell,q)=1}}^Q \left(\frac{\varphi(q)}{q}\cdot\frac{q}{\ell} + O (\sigma_0(q))\right)
=\frac{C(\ell)}{2\ell} Q^2+O_\ell (Q\log Q),
\]
concluding the proof.
\end{proof}

This also establishes the first equality in \eqref{1.2} because
\[
\# \wFF_{Q,\ell}=\# \FF_Q -\# \FF_Q^{(\ell)} \sim \bigg( \frac{1}{2\zeta(2)}-\frac{C(\ell)}{2\ell}\bigg) Q^2.
\]

Letting $\xi>0$ and $Q,\ell\in \N$ with $\ell\geqslant 2$, we set out to asymptotically estimate the number $N^{(\ell)}_Q (\xi)$ as $Q\rightarrow\infty$.
Now if $\gamma=\frac{a}{q}$ and $\gamma^\prime = \frac{a'}{q'}$ are consecutive elements in $\mathcal{F}_Q$ and $\gamma^\prime \in\mathcal{F}_Q^{(\ell)}$,
then $1=a'q-aq'\equiv -aq' \pmod{\ell}$, which implies that $(a,\ell)=1$, and thus $\gamma\notin\mathcal{F}_Q^{(\ell)}$.
Similarly, if $\gamma\in\mathcal{F}_Q^{(\ell)}$, then $\gamma^\prime \notin\mathcal{F}_Q^{(\ell)}$; and so no two consecutive elements of
$\mathcal{F}_Q$ belong simultaneously to $\mathcal{F}_Q^{(\ell)}$. This means that if $\gamma <\gamma^\prime$ are consecutive elements in
$\wFF_{Q,\ell}$, then two cases can occur:

\noindent{\em Case 1.} $\gamma$ and $\gamma^\prime$ are consecutive elements in $\mathcal{F}_Q$ and $\gamma,\gamma^\prime \notin \mathcal{F}_Q^{(\ell)}$.
In this case the number of gaps in consecutive fractions of length $\leqslant \frac{\xi}{Q^2}$ is equal to $\NN_1 (Q,\xi)=N_Q (\xi)-M_1 (Q,\xi)- M_2 (Q,\xi)$,
where $M_1 (Q,\xi)$ is the number of pairs $(\gamma ,\gamma^\prime)\in \GG_Q (\xi)$ with $\gamma^\prime \in\mathcal{F}_Q^{(\ell)}$, and $M_2 (Q,\xi)$
the number of pairs $(\gamma,\gamma^\prime)\in \GG_Q (\xi)$ with $\gamma\in\mathcal{F}_Q^{(\ell)}$.

The number $N_Q (\xi)$ is estimated employing the well-known fact that $\gamma<\gamma^\prime$ are consecutive elements in $\FF_Q$ if and only if
$q,q'\in\{1,\ldots,Q\}$, $q+q'>Q$, and $a'q-aq'=1$. Furthermore, $\frac{a'}{q'} - \frac{a}{q}=\frac{1}{qq'}$, and so
$\frac{a'}{q'}-\frac{a}{q} \leqslant \frac{\xi}{Q^2}$ if and only if $qq'\geqslant \frac{Q^2}{\xi}$. This establishes the equality
\begin{equation}\label{2.1}
\begin{split}
N_Q (\xi) & =\#\bigg\{(q,q')\in\mathbb{N}^2:q,q'\leqslant Q,q+q'>Q,(q,q')=1,qq'\geqslant \frac{Q^2}{\xi}\bigg\} \\
& =\sum_{q'=1}^Q \sum_{\substack{q\in I_Q (q^\prime) \\(q,q')=1}} 1,
\end{split}
\end{equation}
where $I_Q (q^\prime)=Q \cdot \big[ \eta_Q (q^\prime),1\big]$ and $\eta_Q (q^\prime)=\max \big\{1- \frac{q'-1}{Q}, \frac{Q}{\xi q'}\big\}$.

Standard M\" obius summation provides
\[
\begin{split}
N_Q (\xi) & =\sum_{q'=1}^Q\left(\frac{\varphi(q')}{q'}|I_Q (q')|+O(\sigma_0(q'))\right)=\sum_{q'=1}^Q\frac{\varphi(q')}{q'}
|I_Q (q')|+O(Q\log Q) \\ &
=\frac{A(\xi)}{\zeta(2)}Q^2+O(Q\log Q),
\end{split}
\]
where
\begin{equation}\label{2.2}
\begin{split}
A(\xi) & =\bigg| \bigg\{ (x,y) \in (0,1]^2 : x+y>1, xy\geqslant  \frac{1}{\xi}\bigg\} \bigg| \\ &
=\begin{cases} 0 & \mbox{\rm if $0<\xi\leqslant 1$} \\
1-\frac{\log\xi+1}{\xi} & \mbox{\rm if $1\leqslant \xi\leqslant 4$} \\
1-\frac{1}{\xi}-\frac{1}{2}\sqrt{1-\frac{4}{\xi}}+\frac{2}{\xi}\log\Big( \frac{1+\sqrt{1-4/\xi}}{2}\, \Big)
& \mbox{\rm if $\xi \geqslant 4.$}
\end{cases}
\end{split}
\end{equation}

Next, we estimate $M_1 (Q,\xi)$. Clearly $M_1 (Q,\xi)=0$ if $\xi\in(0,1]$, and so assume $\xi>1$. If $\frac{a'}{q'}\in\mathcal{F}_Q^{(\ell)}$,
then $(a',q')=1$ and $\ell\mid a'$. Since $(a^\prime,q^\prime)=1$, we have $(\ell,q')=1$. Therefore, we have to count all
pairs of integers $(q,q')\in (0,Q]^2$ with $q+q'>Q$, $(q,q')=1$, $qq'\geqslant \frac{Q^2}{\xi}$, in which $(\ell,q')=1$, and there is an
$a'\in\{1,\ldots,q'\}$ such that $a'q\equiv1\pmod{q'}$ and $\ell\mid a'$. As a result,
after also letting $k= \frac{a'}{\ell}$, $\ell \overline{\ell} \equiv 1 \pmod{q^\prime}$, $M_1 (Q,\xi)$ can be expressed as
\begin{equation}\label{2.3}
M_1 (Q,\xi)=\sum_{\substack{q'=1\\(\ell,q')=1}}^Q\sum_{\substack{q\in I_Q (q')\\(q,q')=1}}
\sum_{\substack{a'=1\\a'q\equiv1 \hspace{-6pt}\pmod{q'}\\ \ell\mid a'}}^{q'}1
= \sum_{\substack{q'=1 \\ (\ell,q')=1}}^Q\sum_{\substack{q\in I_Q (q')\\ (q,q')=1}} \sum_{\substack{k\in (0,q'/\ell]\\ kq\equiv \overline{\ell}\hspace{-6pt}\pmod{q'}}} 1 .
\end{equation}
Now by \eqref{2.3} and \eqref{A4}, for any $\delta>0$,
\[
\begin{split}
M_1 (Q,\xi) & =\sum_{\substack{q'=1\\(\ell,q')=1}}^Q \left(\frac{\varphi(q')}{q^{\prime ^2}} \iint_{I_Q (q')\times[0,q'/\ell]} dx \, dy + O_\delta \big(q^{\prime 1/2+\delta}\big)\right) \\  &
=\frac{1}{\ell} \sum_{\substack{q'=1 \\ (\ell,q')=1}}^Q \frac{\varphi(q')}{q'} | I_Q (q')|+O_{\ell,\delta} \big( Q^{3/2 +\delta}\big).
\end{split}
\]
Then using \eqref{A2}, we have
\begin{equation*}
\begin{split}
\frac{1}{\ell}\sum_{\substack{q'=1\\(\ell,q')=1}}^Q \frac{\varphi(q')}{q'} |I_Q (q')| &  =
\frac{C(\ell)}{\ell}\int_0^Q |I_Q (q')|\, dq^\prime +O_\ell (Q\log Q) \\ & =\frac{C(\ell)}{\ell}  A(\xi) Q^2+O_\ell (Q\log Q).
\end{split}
\end{equation*}
This proves $M_1 (Q,\xi)\sim \frac{C(\ell)}{\ell} A(\xi) Q^2$ if $\xi>1$. The formula for $M_2 (Q,\xi)$ is analogous and we infer
\begin{equation}\label{2.4}
\begin{split}
\NN_1(Q,\xi) & = N_Q (\xi) - M_1(Q,\xi)-M_2(Q,\xi) \\
& = \bigg( \frac{1}{\zeta(2)} - \frac{2C(\ell)}{\ell} \bigg) A(\xi) Q^2 +O_{\ell,\delta} \big( Q^{3/2+\delta}\big) .
\end{split}
\end{equation}

\noindent{\em Case 2.} There is exactly one fraction in $\mathcal{F}_Q$ between $\gamma$ and $\gamma^\prime$ that belongs to $\mathcal{F}_Q^{(\ell)}$.
It is more convenient to change $\gamma^\prime$ to $\gamma^{\prime\prime}$, so we shall consider triples $\gamma < \gamma^\prime < \gamma^{\prime\prime}$
of elements in $\FF_Q$ with $\gamma^\prime \in \FF_Q^{(\ell)}$ and with $\gamma^{\prime\prime}-\gamma \leqslant \frac{\xi}{Q^2}$.
The equalities
\begin{equation}\label{2.5}
\frac{a^{\prime\prime}+a}{a^\prime}=\frac{q^{\prime\prime}+q}{q^\prime} =K
\quad \mbox{\rm and} \quad \gamma^{\prime\prime}-\gamma =\frac{K}{qq^{\prime\prime}},
\end{equation}
involving the number
\[
K=\nu_2 (\gamma)=\bigg[ \frac{Q+q}{q^\prime}\bigg],
\]
called the \emph{index} of the Farey fraction $\gamma = \frac{a}{q} \in\FF_Q$,
will be useful here. In particular, the inequality $\gamma^{\prime\prime}-\gamma \leqslant \frac{\xi}{Q^2}$ enforces $K\leqslant \xi$.
Consider the set $J_{Q,K,\xi}(q^\prime)$ of elements $q\in (Q-q^\prime ,Q] \cap \big[ Kq^\prime -Q,(K+1)q^\prime -Q\big)$ that satisfy
$\frac{K}{q(Kq^\prime -q)} \leqslant \frac{\xi}{Q^2}$. This set is either empty, an interval, or the union of two intervals.
The number $\NN_2(Q,\xi)$ of gaps of consecutive elements in $\wFF_{Q,\ell}$ of length $\leqslant \frac{\xi}{Q^2}$ that arise in this case
can now be expressed, with $k$ and $\overline{\ell}$ as in \eqref{2.3}, as
\begin{equation}\label{2.6}
\begin{split}
\NN_2 (Q,\xi) & = \sum_{1\leqslant K\leqslant \xi} \sum_{q^\prime \leqslant Q} \sum\limits_{\substack{q\in J_{Q,K,\xi}(q^\prime) \\ (q,q^\prime)=1}}
\sum\limits_{\substack{a^\prime =1 \\ a^\prime q \equiv 1 \hspace{-6pt} \pmod{q^\prime} \\ \ell \mid a^\prime}}^{q^\prime} 1 \\ &
=\sum_{1\leqslant K\leqslant \xi} \sum\limits_{\substack{q^\prime \leqslant Q\\ (\ell,q^\prime)=1}}
\sum\limits_{\substack{q\in J_{Q,K,\xi}(q^\prime) \\ k\in (0,q^\prime /\ell] \\ kq\equiv
\overline{\ell}\hspace{-6pt} \pmod{q^\prime}}} 1.
\end{split}
\end{equation}

We will employ elementary properties of the area preserving invertible transformation $T:\TT \rightarrow\TT$ defined \cite{BCZ} by
\begin{equation}\label{2.7}
T(x,y)=\big( y,\kappa (x,y)y-x\big) ,\qquad (x,y)\in \TT,\qquad \mbox{\rm where}
\end{equation}
\[
\TT =\{ (x,y) \in (0,1]^2: x+y>1 \} \quad \mbox{\rm and} \quad
\kappa (x,y)=\bigg[ \frac{1+x}{y}\bigg].
\]
An important connection with Farey fractions is given by the equality
\begin{equation}\label{2.8}
T\bigg( \frac{q_i}{Q},\frac{q_{i+1}}{Q} \bigg) =\bigg( \frac{q_{i+1}}{Q},\frac{q_{i+2}}{Q}\bigg) .
\end{equation}
For each $K\in \N$ consider the subset $\TT_K=\{ (x,y)\in\TT : \kappa (x,y)=K\}$ of $\TT$,
described by the inequalities $0< x,y\leqslant 1$, $x+y>1$, and $Ky-1\leqslant x< (K+1)y-1$.

Denote $V_{Q,K,\xi}(q^\prime)=\vert J_{Q,K,\xi} (q^\prime)\vert$, so $V_{Q,K,\xi} (Qu)=Q W_{K,\xi} (u)$, where
\[
W_{K,\xi}(u) =\big| \{ v: (v,u)\in \TT_K \} \cap \{ v: K \leqslant \xi v (Ku-v)\} \big| .
\]
Similar arguments as in the proof of \eqref{2.4} lead to
\[
\begin{split}
\NN_2 (Q,\xi) & = \frac{C(\ell)}{\ell} \, Q^2 \sum_{K\leqslant \xi} \int_0^1 W_{K,\xi} (u)\, du +O_{\ell,\delta,\xi} (Q^{3/2+\delta}) \\  &
= \frac{C(\ell)}{\ell} Q^2 \sum_{K\leqslant \xi} A_K (\xi) + O_{\ell,\delta,\xi} (Q^{3/2+\delta}) ,
\end{split}
\]
uniformly in $\xi$ on compact subsets of $[0,\infty)$, where
\[
A_K (\xi) =\operatorname{Area} \big(\Omega_K (\xi)\big), \qquad
\Omega_K (\xi) =\bigg\{ (v,u)\in \TT_K :  u \geqslant f_{K,\xi}(v):=\frac{v}{K}+\frac{1}{\xi v} \bigg\} .
\]

Summarizing, we have shown
\begin{equation*}
N^{(\ell)}_Q (\xi) = G_\ell (\xi) Q^2+O_{\ell,\xi,\delta} (Q^{3/2+\delta}) \qquad (\mbox{\rm as $Q\rightarrow \infty$}),
\end{equation*}
where
\begin{equation}\label{2.9}
G_\ell (\xi) =\bigg( \frac{1}{\zeta(2)} -\frac{2C(\ell)}{\ell} \bigg) A(\xi) +
\frac{C(\ell)}{\ell} \sum_{K\leqslant \xi} A_K (\xi) .
\end{equation}
Taking also into account Lemma \ref{L1} we conclude that the gap limiting measure of
$(\wFF_{Q,\ell})$ exists and its distribution function is given by
\begin{equation*}
\widetilde{F}_\ell (\xi) =\int_0^\xi d\tilde{\nu}_\ell =\frac{1}{\widetilde{K}_\ell} G_\ell \bigg(
\frac{\xi}{\widetilde{K}_\ell} \bigg) .
\end{equation*}

\subsection{Explicit expressions of $A_K (\xi)$}

\subsubsection{$K=1$}
$\TT_1$ is the triangle with vertices $(0,1)$, $(1,1)$, and $\big( \frac{1}{3},\frac{2}{3}\big)$.
When $\xi\leqslant 4$ we have $f_{1,\xi}(v) \geqslant 1$ for every $v>0$, so $A_1 (\xi) =0$. When $\xi >4$ we have
\[
A_1 (\xi) =\int_{u_1}^{u_2} \bigg( 1-\max \bigg\{ f_{1,\xi}(v) ,1-v,\frac{v+1}{2}\bigg\} \bigg) dv,
\]
where $u_{1,2}=\frac{1}{2} \big( 1\pm \sqrt{1-\frac{4}{\xi}}\, \big)$, $0<u_1<u_2<1$, are
the solutions of $f_{1,\xi}(v)=1$. When $4<\xi\leqslant 8$ we have $f_{1,\xi}(v) \geqslant \max\big\{ 1-v,\frac{1+v}{2}\big\}$,
so $A_1 (\xi)$ is the area of the region defined by $v\in [u_1,u_2]$ and $u\in [f_{1,\xi}(v), 1]$.
When $\xi \geqslant 8$ let $v_{1,2}=\frac{1}{4} \big( 1\pm \sqrt{1-\frac{8}{\xi}}\, \big)$, $v_1 <v_2$,
denote the solutions of $f_{1,\xi}(v)=1-v$ and by $w_{1,2}:=2v_{1,2}$ the solutions of $f_{1,\xi}(w)=\frac{w+1}{2}$.
If $8\leqslant \xi\leqslant 9$, then $0<u_1<v_1\leqslant v_2 \leqslant \frac{1}{3} \leqslant w_1 \leqslant w_2 <u_2 <1$.
In this case $A_1(\xi)$ is the area of the region described by
$v\in [u_1,v_1] \cup [v_2,w_1]\cup [w_2,u_2]$ and $u\in [f_{1,\xi}(v),1]$,
$v\in [v_1,v_2]$ and $u\in [1-v,1]$, or $v\in [w_1,w_2]$ and
$u\in \big[ \frac{1+v}{2},1\big]$. Finally, if $\xi>9$, then $0<u_1<v_1<w_1 <\frac{1}{3}<v_2<w_2<u_2<1$, and
$A_1(\xi)$ is the area of the region described by $v\in [u_1,v_1] \cup [w_2,u_2]$ and $u\in [f_{1,\xi}(v),1]$,
or $v\in \big[ v_1,\frac{1}{3}\big]$ and $u\in [1-v,1]$, or $v\in \big[ \frac{1}{3},w_2]$ and $u\in\big[ \frac{1+v}{2},1\big]$.
A plain calculation gives
\[
A_1 (\xi) =\begin{cases} 0 & \mbox{\rm if $0<\xi\leqslant 4$} \\
\frac{1}{2}\sqrt{1-\frac{4}{\xi}} -\frac{1}{\xi} \ln \big( \frac{u_2}{u_1}\big) & \mbox{\rm if $4\leqslant \xi \leqslant 8$} \\
\frac{1}{2}\sqrt{1-\frac{4}{\xi}} -\frac{1}{\xi} \ln \big( \frac{u_2}{u_1}\big) - \frac{1}{2} \sqrt{1-\frac{8}{\xi}}
+\frac{2}{\xi} \ln \big( \frac{v_2}{v_1}\big)  & \mbox{\rm if $8\leqslant \xi \leqslant 9$} \\
\frac{1}{2}\sqrt{1-\frac{4}{\xi}} - \frac{1}{\xi} \ln \big( \frac{u_2}{u_1}\big)
- \frac{1}{4} \sqrt{1-\frac{8}{\xi}} -\frac{1}{12} +\frac{1}{\xi} \ln \big( \frac{2v_2}{v_1}\big)
& \mbox{\rm if $\xi \geqslant 9$.}
\end{cases}
\]

\begin{figure}[ht]
\begin{center}
\includegraphics*[scale=0.32, bb=0 0 360 380]{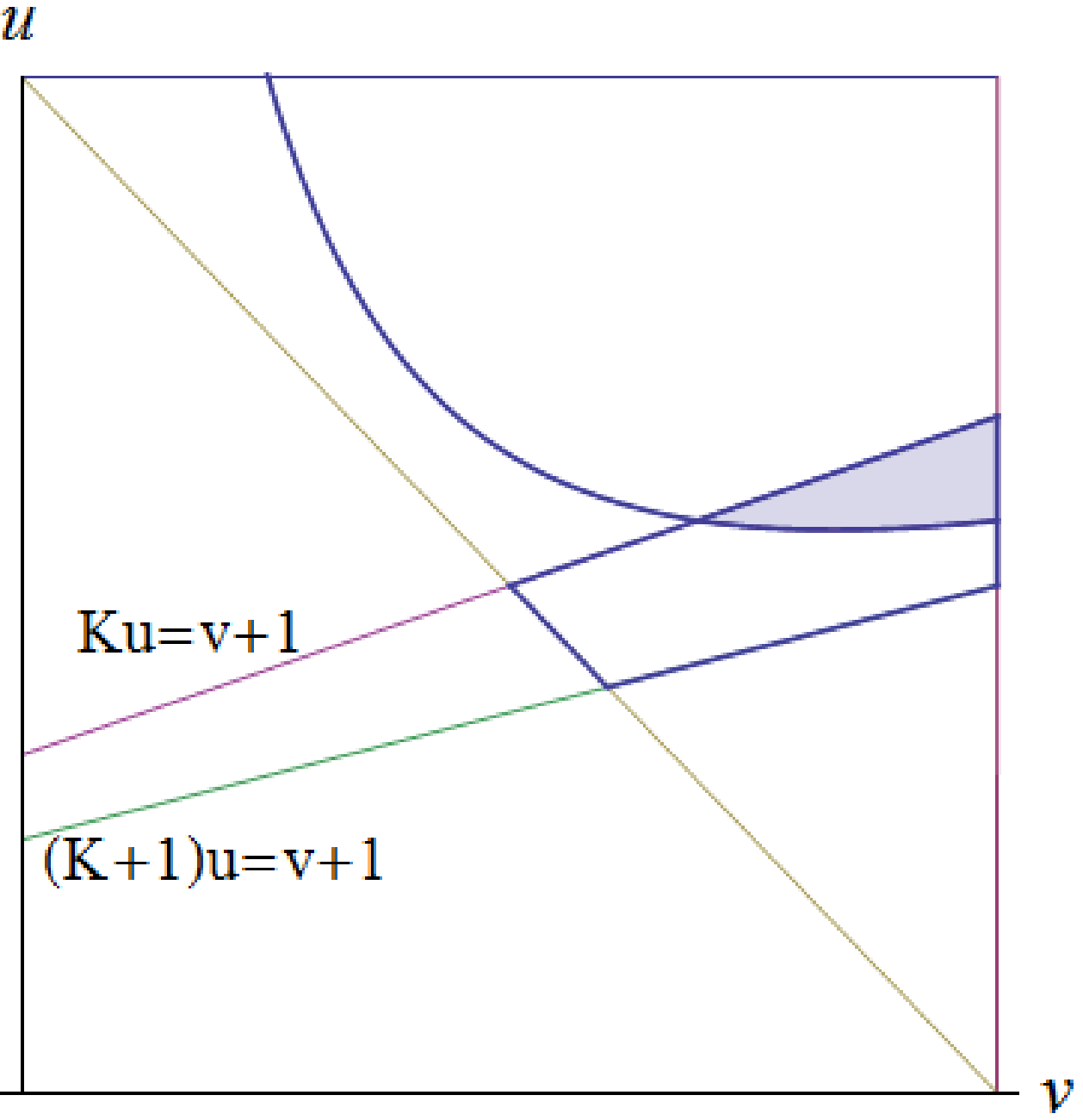}
\includegraphics*[scale=0.32, bb=0 0 360 380]{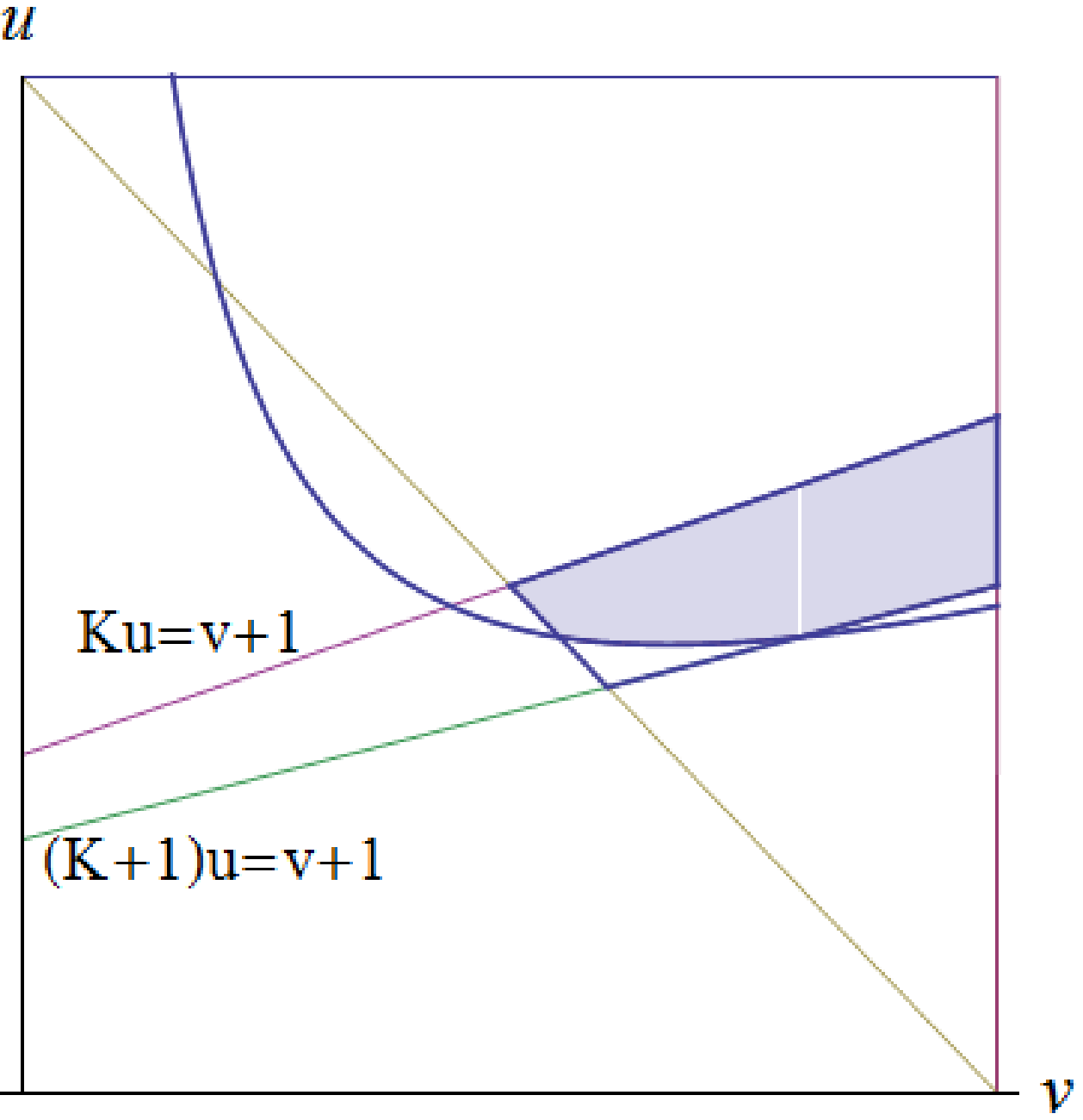}
\includegraphics*[scale=0.32, bb=0 0 360 380]{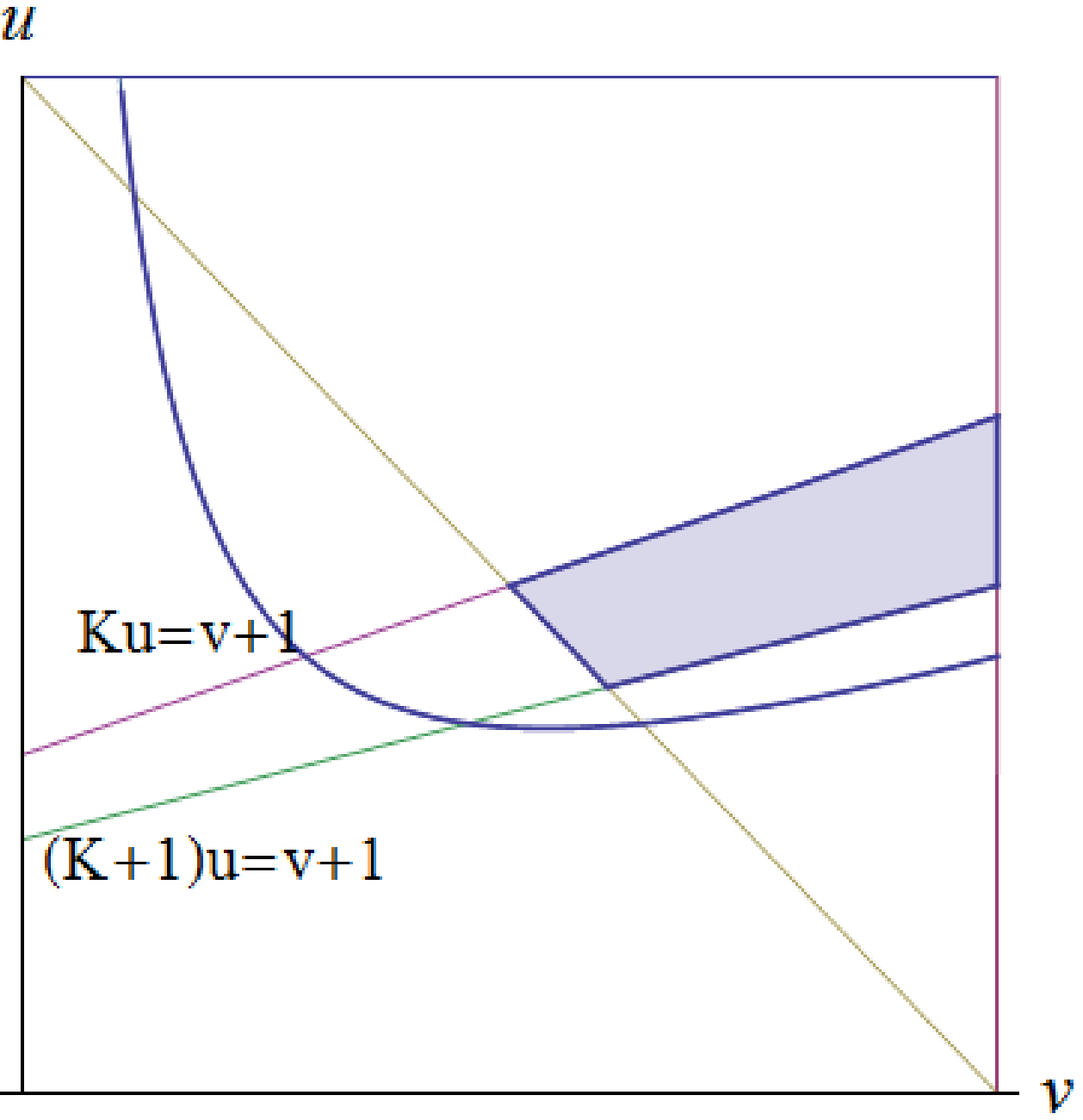}
\caption{The intersection between the quadrilateral $\TT_K$ and the curve $u=f_{K,\xi}(v)$ when $K< \xi <\frac{K(K+1)}{K-1}$,
$\frac{K(K+1)}{K-1} \leqslant \xi < \frac{(K+2)^2}{K}$, and respectively $\xi\geqslant \frac{(K+2)^2}{K}$}
\label{Figure1}
\end{center}
\end{figure}

\subsubsection{$K\geqslant 2$} Note that $f_{K,\xi}(1)=f_{K,\xi}\big( \frac{K}{\xi}\big)=\frac{1}{K}+\frac{1}{\xi}$.
The situation is described by Figure \ref{Figure1}.
The solution of $f_{K,\xi}(v)=\frac{v+1}{K}$ is $v=\frac{K}{\xi}$, so the curve $u=f_{K,\xi}(v)$ intersects the
upper edge of $\TT_K$ if and only if $K<\xi < \frac{K(K+1)}{K-1}$, in which case it does not intersect the two lower edges of $\TT_K$ and
\[
A_K (\xi)=\int_{K/\xi}^1 \bigg( \frac{v+1}{K}-f_{K,\xi}(v)\bigg) dv
=\int_{K/\xi}^1 \bigg( \frac{1}{K}-\frac{1}{\xi v}\bigg)\, dv .
\]
The solution of $f_{K,\xi}\big( \frac{K}{K+2}\big) > \frac{2}{K+2}$ is $\xi < \frac{(K+2)^2}{K}$. This shows that when
$\frac{K(K+1)}{K-1} \leqslant \xi < \frac{(K+2)^2}{K}$ the graph of $u=f_{K,\xi}(u)$ intersects the segment $u=1-v$, $v\in \big[
\frac{K-1}{K+1},\frac{K}{K+2}\big]$, exactly when $v=v_K=\frac{K}{2(K+1)} \Big( 1+\sqrt{1-\frac{4}{\xi} (1+\frac{1}{K})}\,\Big)$,
and the segment $u=\frac{v+1}{K+1}$, $v\in \big[ \frac{K}{K+2},1\big]$, exactly at $v=w_K=\frac{K}{2} \Big( 1-\sqrt{1-\frac{4}{\xi} \big( 1+\frac{1}{K}\big)}\,\Big)$,
so in this case
\[
A_K (\xi) =\operatorname{Area}(\TT_K) -\int_{v_K}^{w_K} f_{K,\xi}(v)\, dv +\int_{v_K}^{K/(K+2)} (1-v)\, dv +
\int_{K/(K+2)}^{w_K} \frac{v+1}{K+1}\, dv .
\]
Finally, when $\xi > \frac{(K+2)^2}{K}$, the graph of $u=f_{K,\xi}(v)$ does not intersect any of the
edges of $\TT_K$ and
\[
A_K (\xi) =\operatorname{Area} (\TT_K) .
\]
In summary, a quick calculation leads to
\begin{equation*}
A_K (\xi)=\begin{cases}
0 & \mbox{\rm if $0\leqslant \xi\leqslant K$} \\
\frac{1}{K}-\frac{1}{\xi}-\frac{1}{\xi} \ln \big( \frac{\xi}{K}\big)
& \mbox{\rm if $K\leqslant\xi\leqslant \frac{K(K+1)}{K-1}$} \\
\frac{K^3+8}{2K(K+1)(K+2)}-\frac{1}{\xi} \ln \big( \frac{w_K}{v_K}\big) -\frac{v_K}{2} +\frac{w_K}{2(K+1)}
& \mbox{\rm if $\frac{K(K+1)}{K-1} \leqslant \xi \leqslant \frac{(K+2)^2}{K}$} \\
\frac{4}{K(K+1)(K+2)} & \mbox{\rm if $\xi \geqslant \frac{(K+2)^2}{K}$.}
\end{cases}
\end{equation*}

\begin{figure}[ht]
\begin{center}
\includegraphics*[scale=0.6, bb=0 0 280 170]{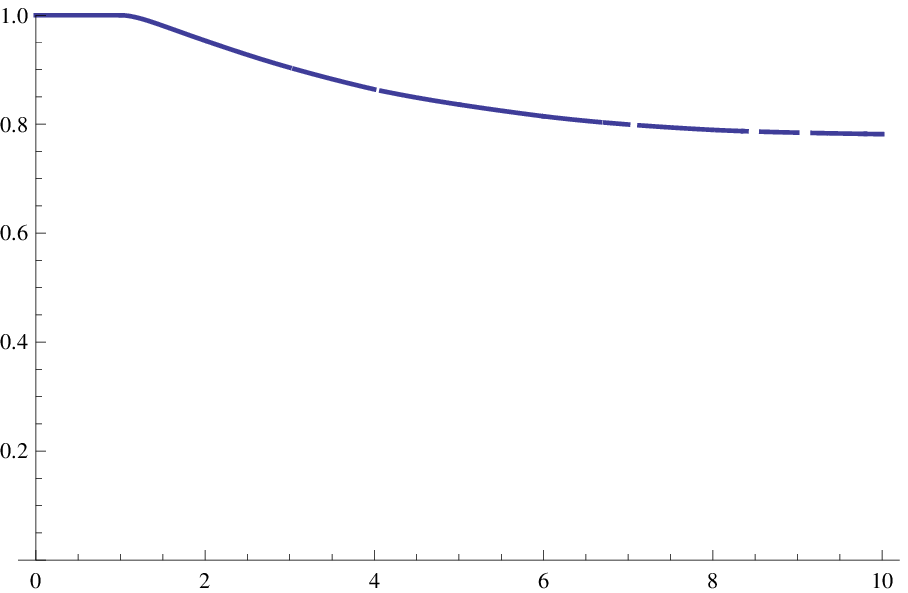}
\includegraphics*[scale=0.65, bb=0 0 280 160]{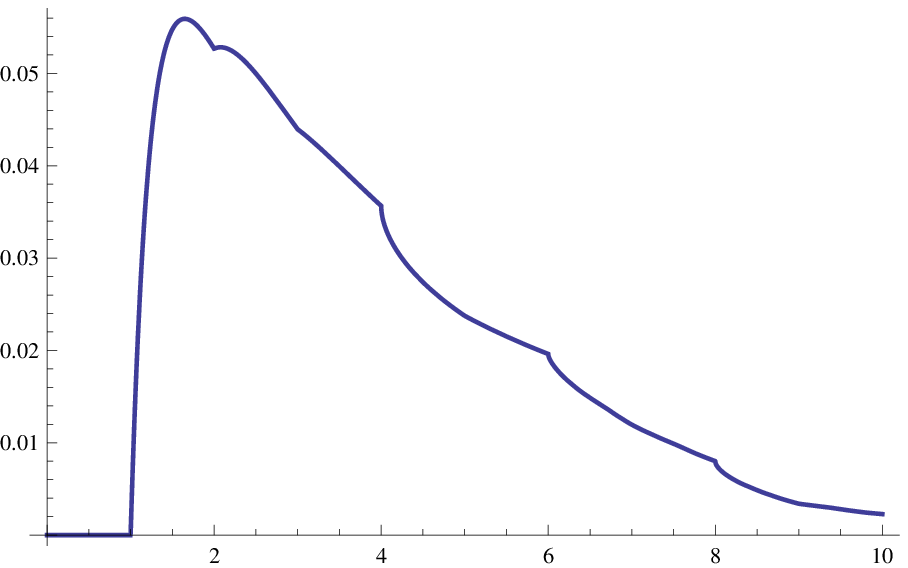}
\caption{The repartition function $1-G_3(\xi)$ and the density $-G_3^\prime (\xi)$}
\label{Figure2}
\end{center}
\end{figure}

\section{Consecutive elements in $\FF_Q$ with denominator relatively prime to $d$}
In this section we comment on the first two steps in the proof of \eqref{1.1} from \cite{BH}.

\subsection{Upper bounds on the number of consecutive Farey fractions whose denominators are not relatively prime to $d$}
One of the key steps in the proof of \eqref{1.1} in \cite{BH} is to show that
for any $Q$ and any $d$, any string of consecutive elements in $\FF_Q$ of length $4d^3$ contains at least one element whose denominator is coprime with $d$.
Next we provide two arguments which show that the upper bound $L(d)$ should actually be much smaller than $4d^3$.

\begin{lemma}\label{L2}
If $\omega (d) \leqslant \min \{ p\in {\mathcal P}: p\mid d\}$, then $L(d)\leqslant 4\omega (d)^3$.
\end{lemma}

\begin{proof}
We first revisit the proof of the first part of Step (i) in the proof of Theorem 1 in \cite{BH} (pp. 210--211).
Suppose $Q$ and $i_1<i_2$ are chosen such that, for every $j\in [i_1,i_2]$,
\[
\max\{ q_{i_1},q_{i_2}\} \leqslant q_j \quad \mbox{\rm and} \quad
(q_j,d)>1.
\]
Then $(q_{i_1},q_{i_2})=1$ and
\begin{equation}\label{3.1}
\{ q_j :i_1 <j<i_2\} \subset \{ mq_{i_1}+nq_{i_2} : m,n\in {\mathbb N} , (m,n)=1, mq_{i_1}+nq_{i_2} \leqslant Q\} .
\end{equation}

Let $d_1=p_1^{\alpha_1} \cdots p_\omega^{\alpha_\omega}$, with $p_1 <\cdots < p_\omega$ primes, be the largest divisor of $d$ which is coprime to $q_{i_1}$.
Then $\omega <\omega (d) \leqslant \min\{ p\in {\mathcal P}: p\mid d\} \leqslant p_1$. Fix some integer $L$ with $\omega +1 \leqslant L \leqslant p_1$.
We claim that there exists $m_1\in\N$, $m_1 \leqslant L$ such that $(m_1 q_{i_1}+q_{i_2},d_1)=1$. If not, then
$(\ell q_{i_1}+q_{i_2}, d_1)>1$ for all $\ell \in \{ 1,\ldots,L\}$. Since $L>\omega$, the Pigeonhole Principle shows that there exist
$i_0 \in \{ 1,\ldots,\omega\}$ and $1\leqslant \ell < \ell^\prime \leqslant L$ such that $p_{i_0} \mid (\ell q_{i_1}+q_{i_2})$ and
$p_{i_0} \mid (\ell^\prime q_{i_1}+q_{i_2})$, and so $p_{i_0} \mid (\ell^\prime -\ell) q_{i_1}$. But $(p_{i_0},q_{i_1})=1$,
hence $L >\ell^\prime -\ell \geqslant p_{i_0} \geqslant p_1$, which contradicts $L\leqslant p_1$.

So if $(m_1 q_{i_1}+q_{i_2},d)>1$, then there exists $p$ prime with $p\mid q_{i_1}$ and $p\mid (m_1 q_{i_1}+q_{i_2})$, thus
contradicting $(q_{i_1},q_{i_2})=1$. Hence $(m_1 q_{i_1}+q_{i_2},d)=1$, which in turn yields $Q\leqslant m_1 q_{i_1} +q_{i_2} \leqslant
L q_{i_1} +q_{i_2}$. In a similar way one has $Q\leqslant q_{i_1}+L q_{i_2}$, thus \eqref{3.1} leads to
\[
\{ q_j : i_1 <j<i_2\} \subset \{ mq_{i_1}+nq_{i_2}: 1\leqslant m,n\leqslant L\},
\]
and in particular $i_2-i_1 \leqslant L^2$.

The second part of the proof proceeds ad litteram as in the proof of Step (i) \cite[pp.~211--212]{BH} replacing $d$ there by $L$.
\end{proof}

When $d$ is the product of two prime powers the bound above can be lowered. In this case we show that $L(d)\leqslant 5$, which
is sharp for $d=6$ because $\frac{1}{4} < \frac{1}{3} < \frac{1}{2} < \frac{2}{3} < \frac{3}{4}$ are consecutive in $\FF_{4}$.
Our proof employs elementary properties of the transformation $T$ from \eqref{2.7}. In particular \eqref{2.8} and the
following inclusions will be useful in the proof of Lemma \ref{L3}:
\[
\begin{split}
 & T \TT_k \subseteq \TT_1\ \ \mbox{\rm if $k\geqslant 5$}, \qquad \qquad \   T (\TT_3\cup\TT_4) \subseteq \TT_1 \cup \TT_2,\\ &
T\TT_2 \subseteq \TT_1 \cup \TT_2 \cup \TT_3 \cup \TT_4,\qquad T(T\TT_3 \cap \TT_2) \subseteq \TT_1 \cap \TT_2 .
\end{split}
\]

\begin{lemma}\label{L3}
If $d=p^\alpha q^\beta$, then for each $i \in \{ 0,\ldots ,\# \FF_Q -5\}$ there exists $j\in \{ 0,\ldots ,5\}$ such that $(q_{i+j},d)=1$,
and so $L(d)\leqslant 5$.
\end{lemma}

\begin{proof}
We have $q_{i+2}=Kq_{i+1}-q_i$, $q_{i+3}=K^\prime q_{i+2}-q_{i+1}$, $q_{i+4}=K^{\prime\prime} q_{i+3}-q_{i+2}$, $q_{i+5}=K^{\prime\prime\prime} q_{i+4}-q_{i+3}$,
where $K=\kappa \big( \frac{q_i}{Q},\frac{q_{i+1}}{Q}\big)$, $K^\prime=\kappa \big(\frac{q_{i+1}}{Q},\frac{q_{i+2}}{Q}\big)$,
$K^{\prime\prime}=\kappa \big(\frac{q_{i+2}}{Q},\frac{q_{i+3}}{Q}\big)$, and $K^{\prime\prime\prime}=\kappa \big(\frac{q_{i+3}}{Q},\frac{q_{i+4}}{Q}\big)$.
Suppose that $(q_{i},d),\ldots ,(q_{i+5},d)>1$. Then either $p\mid (q_{i},q_{i+2},q_{i+4})$ and
$q\mid (q_{i+1},q_{i+3},q_{i+5})$, or vice versa.

Without loss of generality we can work in the first case. The equality $q_{i+2}+q_i=Kq_{i+1}$ and $p\nmid q_{i+1}$ yield $p\mid K$.
Similarly we have $q\mid K^\prime$.
Assume first that $K\geqslant 5$.
Since $\big( \frac{q_i}{Q},\frac{q_{i+1}}{Q}\big) \in \TT_K$ and $T\TT_K \subseteq \TT_1$ we must have $K^\prime=1$,which contradicts $q\geqslant 2$. In particular $p\geqslant 5$ cannot occur.

When $p=3$ and $K=3$, from $T\TT_3 \subseteq \TT_1 \cup \TT_2$ it follows that
$K^\prime \in \{ 1,2\}$. Since $q\mid K^\prime$, we infer $q=2$.
The region $T\TT_3 \cap \TT_2$ is the quadrilateral with vertices at
$\big( \frac{1}{2},\frac{1}{2}\big)$, $\big(\frac{2}{5},\frac{3}{5}\big)$, $\big( \frac{3}{5},\frac{4}{5}\big)$, and
$\big( \frac{3}{7},\frac{5}{7}\big)$, being further mapped by $T$ into a subset of $\TT_1 \cup \TT_2$ whence $K^{\prime\prime} \in \{ 1,2\}$.
Again $K^{\prime\prime}=1$ leads to an immediate contradiction, while $K^{\prime\prime}=2$ yields
$q_{i+2}+q_{i+4}=2q_{i+3}$, showing that $p=2$, another contradiction.

When $p=2$ and $K<5$, we have $K\in \{ 2,4\}$. Assume first $K=2$. As $T\TT_2 \subseteq \TT_1 \cup \TT_2 \cup \TT_3 \cup \TT_4$ and
$K^\prime \neq 1$ it remains that $K^\prime \in \{ 2,3,4\}$. Since $q\geqslant 3$ divides $K^\prime$, we infer
$q=3$. Furthermore, $T\TT_3 \subseteq \TT_1 \cup \TT_2$ and $K^{\prime\prime}\neq 1$
yield $K^{\prime\prime}=2$. Employing again $T(T\TT_3 \cap \TT_2) \subseteq \TT_1 \cup \TT_2$, we infer $K^{\prime\prime\prime}=2$, and so
$q_{i+3}+q_{i+5}=2q_{i+4}$. This is again a contradiction, because $3$ divides $q_{i+3}+q_{i+5}$ and cannot divide $2q_{i+4}$.
Finally, assume $K=4$, so $K^\prime \in \{ 1,2\}$, which is not possible because $q\geqslant 3$ divides $K^\prime$.
\end{proof}

Note that if $(p_n)$ is the sequence of primes, then none of the denominators of the fractions in $\FF_{p_n}\setminus\{1\}$ are relatively prime to $\prod_{i=1}^np_i$. This gives the lower bound $\#\FF_{p_n}-1$ on the size of the largest string of consecutive fractions in $\FF_Q\setminus\FF_{Q,d}$ for some $Q,d\in\N$ with $\omega(d)=n$. Since $p_n\sim n\log n$ as $n\rightarrow\infty$ and $\#\FF_{Q}\sim\frac{3}{\pi^2}Q^2$ as $Q\rightarrow\infty$, there exists $A>0$ such that $\#\FF_{p_n}-1\geqslant A(n\log n)^2$. Thus any upper bound on $L(d)$ involving only $\omega(d)$ must be greater than $A(\omega(d)\log\omega(d))^2$.

\subsection{The index and the continuant}
The second step in the proof of (1.1) in \cite{BH} relies on \cite[Lemma 1]{BH}, which is actually exactly Remark 2.6 in \cite{BGZ} (see also \cite[Lemma 5]{BCZ}), and on
a result relating the $\ell$-index of a Farey fraction and the continuant of regular continued fractions.
The $\ell$-index of $\gamma_i=\frac{a_i}{q_i}\in\FF_Q$ is the positive integer
$\nu_\ell(\gamma_i)=a_{i+\ell-1}q_{i-1}-a_{i-1}q_{i+\ell-1}$ where $\frac{a_{i+k}}{q_{i+k}}$ denotes the
$k^{\mathrm{th}}$ successor of $\gamma_i$ in $\FF_Q$.
The (regular continued fraction) \emph{continuants} are defined as usual by
$K_0(\cdot)=1$, $K_1(x_1)=1$, and
\[
K_\ell (x_1,\ldots,x_\ell)=x_\ell K_{\ell-1}(x_1,\ldots,x_{\ell-1}) +K_{\ell-2}(x_1,\ldots,x_{\ell-2})\quad
\mbox{\rm if $\ell\geqslant 2$.}
\]
In \cite{H} the identity
\begin{equation}\label{3.2}
\nu_\ell (\gamma_i)=\epsilon_\ell K_{\ell -1} \big( -\nu_2(\gamma_i),\nu_2 (\gamma_{i+1}),\ldots ,
(-1)^{\ell -1} \nu_2 (\gamma_{i+\ell-2})\big)
\end{equation}
was proved, with $\epsilon_\ell=1$ if $\ell\in \{ 0,1\} \pmod{4}$ and
$\epsilon_\ell=-1$ if $\ell \in \{ 2,3\} \pmod{4}$.

We give a very short proof of \eqref{3.2}. We define the
\emph{Farey continuants} $K_\ell^F$ by
$K_0^F(\cdot)=1$, $K_1^F (x_1)=x_1$, and
\[
K_\ell^F (x_1,\ldots,x_\ell) =x_\ell K_{\ell-1}^F (x_1,\ldots,x_{\ell-1})-K_{\ell-2}^F(x_1,\ldots,x_{\ell-2}) \quad
\mbox{\rm if $\ell\geqslant 2$.}
\]
The defining equalities for $K_\ell$ and $K_\ell^F$ plainly yield, for all $\ell\geqslant 2$,
\begin{equation}\label{3.3}
\left( \begin{matrix} x_1 & 1 \\ 1 & 0 \end{matrix}\right) \cdots
\left( \begin{matrix} x_\ell & 1 \\ 1 & 0 \end{matrix}\right) =
\left( \begin{matrix} K_\ell (x_1,\ldots,x_\ell) & K_{\ell-1} (x_1,\ldots,x_{\ell-1}) \\
K_{\ell-1} (x_2,\ldots,x_\ell) & K_{\ell-2} (x_2,\ldots ,x_{\ell-1})\end{matrix}\right) ,
\end{equation}
\begin{equation}\label{3.4}
\left( \begin{matrix} x_1 & 1 \\ -1 & 0 \end{matrix}\right) \cdots
\left( \begin{matrix} x_\ell & 1 \\ -1 & 0 \end{matrix}\right) =
\left( \begin{matrix} K_\ell^F (x_1,\ldots,x_\ell) & K_{\ell-1}^F (x_1,\ldots,x_{\ell-1}) \\
-K_{\ell-1}^F (x_2,\ldots,x_\ell) & -K_{\ell-2}^F (x_2,\ldots ,x_{\ell-1})\end{matrix}\right) .
\end{equation}
From \eqref{3.4} and the definition of $\nu_\ell (\gamma_i)$ we now infer
\begin{equation}\label{3.5}
\nu_\ell (\gamma_i) =  K_{\ell -1}^F \big( \nu_2(\gamma_i),\nu_2 (\gamma_{i+1}),\ldots ,
\nu_2 (\gamma_{i+\ell-2})\big).
\end{equation}
The equality \eqref{3.2} follows immediately from \eqref{3.3}, \eqref{3.4}, \eqref{3.5} and
\[
\left( \begin{matrix} x & 1 \\ -1 & 0 \end{matrix}\right) \left( \begin{matrix} y & 1 \\ -1 & 0 \end{matrix}\right) =
- \left( \begin{matrix} -x & 1 \\ 1 & 0 \end{matrix}\right) \left( \begin{matrix} y & 1 \\ 1 & 0 \end{matrix}\right) .
\]

\section{The gap distribution of $\FF_{Q,d}$}

Letting $d\in\mathbb{N}$ and $\xi>0$, we wish to asymptotically estimate the number of pairs of consecutive elements $\gamma<\gamma'$ in $\mathcal{F}_{Q,d}$ with
$\gamma'-\gamma \leqslant \frac{\xi}{Q^2}$ as $Q\rightarrow\infty$. It is plain that
\[
\# \FF_{Q,d} =\sum\limits_{\substack{q=1 \\ (q,d)=1}}^Q \varphi (q) = C(d) \int_0^Q q\, dq +O_d (Q\log Q) =\frac{C(d)}{2} Q^2 +O_d (Q\log Q),
\]
showing the second equality in \eqref{1.2}.
Denote $N_Q=\# \FF_Q$ and $\gamma_j=\frac{a_j}{q_j}$, so the number of pairs of fractions we wish to estimate is
\[
\begin{split}
N_d (Q,\xi) & =\sum_{\ell=1}^{L(d)} \# \left\{ i\in [1,N_Q] : \begin{matrix} \frac{\nu_\ell(\gamma_i)}{q_{i-1}q_{i+\ell-1}} \leqslant \frac{\xi}{Q^2},\quad
(q_{i-1},d)=(q_{i+\ell-1},d)=1 \\
(q_i,d)>1,\ldots ,(q_{i+\ell-2},d)>1 \end{matrix} \right\} \\
& = \sum_{\ell=1}^{L(d)} \sum_{k=1}^{[\xi]} \# \left\{ i\in [1,N_Q] :
\begin{matrix}  \frac{k}{q_{i-1}q_{i+\ell-1}} \leqslant \frac{\xi}{Q^2}, \quad \nu_\ell(\gamma_i)=k \\ (q_{i-1},d)=(q_{i+\ell-1},d)=1 \\
(q_i,d)>1,\ldots, (q_{i+\ell-2},d) >1  \end{matrix} \right\}.
\end{split}
\]
It is shown in \cite{BCZ,BG} that given $i\in [1,N_Q]$ and $k,\ell\in\mathbb{N}$ with $\ell\geqslant 2$,
if $\nu_\ell(\gamma_i)=k$, then the $(\ell-1)$-tuple $\big( \nu_2(\gamma_i),\ldots,\nu_2 (\gamma_{i+\ell-2})\big)$ can take on $n(k,\ell)$ values,
where $n(k,\ell) \in\mathbb{N}\cup\{0\}$ depends only on $k$ and $\ell$ and not on $i$ or $Q$; and in \cite{H}, it is proven that $\nu_\ell(\gamma_i)$ can be determined if $\big(\nu_2(\gamma_i),\ldots,\nu_2(\gamma_{i+\ell-2})\big)$ is known (cf.~identity \eqref{3.2} above). Therefore, letting $\{x(k,\ell,m) \}_{m=1}^{n(k,\ell)}$ be the
$(\ell-1)$-tuples for which $\nu_\ell(\gamma_i)=k$ whenever $x(k,\ell,m)=\big( \nu_2(\gamma_i),\ldots,\nu_2(\gamma_{i+\ell-2})\big)$ for some $m\in \{1,\ldots,n(k,\ell)\}$, we have
\begin{align*}
& N_d (Q,\xi) =\#\bigg\{ i\in [1,N_Q] : (q_{i-1},d)=(q_i,d)=1,\quad q_{i-1}q_i \geqslant \frac{Q^2}{\xi}\bigg \} \\
& \quad + \sum_{\ell=2}^{L(d)} \sum_{k=1}^{[\xi]} \sum_{m=1}^{n(k,\ell)} \#
\left\{ i\in [1,N_Q] : \begin{matrix} q_{i-1} q_{i+\ell-1} \geqslant \frac{kQ^2}{\xi},\  (q_{i-1},d)=(q_{i+\ell-1},d)=1 \\
(q_i,d)>1,\ldots,(q_{i+\ell-2},d) >1 \\ x(k,\ell,m) = (\nu_2(\gamma_i),\ldots,\nu_2(\gamma_{i+\ell-2})) \end{matrix} \right\} .
\end{align*}
Since $q_{j+1}=\nu_2(\gamma_j)q_j-q_{j-1}$ for $j\in [1, N_Q-1]$, the residue classes of the denominators $q_{i-1},\ldots,q_{i+\ell-1}$ can be
determined once the residue classes of $q_{i-1}$ and $q_i$, and the $(\ell-1)$-tuple $\big( \nu_2(\gamma_i),\ldots,\nu_2(\gamma_{i+\ell-2})\big)$ are known.
Thus, there is a subset $\mathcal{A}_{k,\ell,m}\subseteq\{1,\ldots,d\}^2$ such that when $\big( \nu_2(\gamma_i),\ldots,\nu_2(\gamma_{i+\ell-2})\big)=x(k,\ell,m)$, we have $(q_{i-1},d)=(q_{i+\ell-1},d)=1$ and $(q_{i+j-1},d)>1$ for $1\leqslant j<\ell$ if and only if $(q_{i-1},q_i)\pmod{d}\in\mathcal{A}_{k,\ell,m}$. (Note clearly that $(a,d)=1$ for $(a,b)\in\mathcal{A}_{k,\ell,m}$.) Furthermore, if we let $x(k,\ell,m)= \big( x_1(k,\ell,m),\ldots,x_{\ell-1}(k,\ell,m)\big)$
and denote $\mathbb{Z}_{\textrm{vis}}^2=\{(a,b)\in\Z^2 :(a,b)=1\}$, it is clear that $\big( \nu_2(\gamma_i) , \ldots , \nu_2 (\gamma_{i+\ell-2})) = x(k,\ell,m)$ if and only if
$$(q_{i-1},q_i)\in Q\cdot (\mathcal{T}_{x_1(k,\ell,m)}\cap T^{-1}\mathcal{T}_{x_2(k,\ell,m)}\cap\cdots\cap T^{-(\ell-2)}\mathcal{T}_{x_{\ell-1}(k,\ell,m)})\cap\mathbb{Z}_{\textrm{vis}}^2.$$

Now if we let $\pi_1, \pi_2 :\R^2\rightarrow \R$ be the canonical projections, then
\[
\frac{q_{i-1}q_{i+\ell-1}}{Q^2} = \pi_1 \bigg( \frac{q_{i-1}}{Q} , \frac{q_i}{Q}\bigg) \cdot (\pi_2\circ T^{\ell-1}) \bigg( \frac{q_{i-1}}{Q},\frac{q_i}{Q}\bigg) ,
 \]
and so
\[
q_{i-1}q_{i+\ell-1} \geqslant \frac{kQ^2}{\xi} \quad \Longleftrightarrow \quad (q_{i-1},q_i)\in Qg_\ell^{-1} \bigg[ \frac{k}{\xi},\infty \bigg),
\]
where $g_\ell=\pi_1\cdot(\pi_2\circ T^{\ell-1})$.
Now set
$g_1(x,y)=xy$ and
\begin{align*}
\Omega_{k,\ell,m}(\xi) & =\TT_{x_1(k,\ell,m)}\cap T^{-1}\TT_{x_2(k,\ell,m)} \cap \cdots\cap  T^{-(\ell-2)}\TT_{x_{\ell-1}(k,\ell,m)} \cap g_\ell^{-1} \bigg[ \frac{k}{\xi},\infty \bigg),\\
\Omega_1(\xi)&=\mathcal{T}\cap g_1^{-1} \bigg[ \frac{1}{\xi},\infty \bigg),\qquad
\mathcal{A}_1 = \{(a,b):a,b\in [1,d],(a,d)=(b,d)=1\}.
\end{align*}
We then have
\[
\begin{split}
N_d (Q,\xi) = & \sum_{(a,b)\in\mathcal{A}_1} \# Q\Omega_1(\xi) \cap \big( (a,b)+d\Z^2\big) \cap \Z_{\textrm{vis}}^2 \\
& + \sum_{\ell=2}^{L(d)} \sum_{k=1}^{[\xi]} \sum_{m=1}^{n(k,\ell)} \sum_{(a,b)\in\mathcal{A}_{k,\ell,m}}
\# Q\Omega_{k,\ell,m}(\xi) \cap \big( (a,b)+d\Z^2\big) \cap \Z_{\textrm{vis}}^2,
\end{split}
\]
where we have used the fact that if $(a,b) \in Q\mathcal{T} \cap \Z_{\textrm{vis}}^2$, then there is an $i$ such that $a=q_{i-1}$ and $b=q_{i}$. One can prove in a similar manner to \cite[Lemma 2]{BH} that for all bounded $\Omega\subseteq\mathbb{R}^2$ whose boundary can be covered by the images of finitely many Lipschitz functions from $[0,1]$ to $\mathbb{R}^2$, and for all $\mathcal{A}\subseteq\{1,\ldots,d\}^2$ in which $(a,d)=1$ for all $(a,b)\in\mathcal{A}$, we have
$$\sum_{(a,b)\in\mathcal{A}} \# Q\Omega \cap \big((a,b) + d \Z^2\big) \cap \Z_{\textrm{vis}}^2 = \frac{\operatorname{Area}(\Omega) \#\AA}{\zeta(2) d^2}
\prod\limits_{\substack{p\in {\mathcal P} \\ p \mid d}} \bigg( 1-\frac{1}{p^2}\bigg)^{-1} Q^2 + O_d (Q\log Q)$$ as $Q\rightarrow\infty$.
It is easily seen that the boundaries of $\Omega_1(\xi)$ and $\Omega_{k,\ell,m}(\xi)$ can be covered by finitely many Lipschitz functions from $[0,1]$ to $\mathbb{R}^2$, and so we have

\[
N_d (\xi,Q) = C_d (\xi) Q^2 + O_d (Q\log Q) ,
\]
where
\begin{equation*}
\begin{split}
C_d (\xi) & = \frac{1}{\zeta(2) d^2} \prod\limits_{\substack{p\in {\mathcal P} \\ p\mid d}} \bigg( 1-\frac{1}{p^2} \bigg)^{-1} \\
& \hspace{1cm} \cdot
\left( \varphi(d)^2 \operatorname{Area} (\Omega_1(\xi)) + \sum_{\ell=2}^{L(d)} \sum_{k=1}^{[\xi]} \sum_{m=1}^{n(k,\ell)}
\operatorname{Area}(\Omega_{k,\ell,m}(\xi))\#\mathcal{A}_{k,\ell,m}\right),
\end{split}
\end{equation*}
noting that $\#\mathcal{A}_1=\varphi(d)^2$.

The gap limiting measure of
$(\FF_{Q,d})_Q$ exists with distribution function given by
\begin{equation*}
F_d (\xi) =\int_0^\xi d\nu_d =\frac{1}{K_d} C_d \bigg( \frac{\xi}{K_d} \bigg) .
\end{equation*}
When $d$ is a prime power this can be expressed more explicitly as in \eqref{2.9}.

\appendix
\numberwithin{equation}{section}
\section{Appendix}

For the convenience of the reader we collect in this appendix the asymptotic formulas used in this paper.

Assuming that $f$ is a $C^1$ function on the interval of integration in \eqref{A1}-\eqref{A3} and that $I$, $J$ are intervals and $f\in C^1 (I\times J)$
in \eqref{A4}, we have
\begin{equation}\label{A1}
\sum\limits_{\substack{a<k\leqslant b \\ (k,q)=1}} f(k) = \frac{\varphi(q)}{q} \int_a^b f(x)\, dx +
O\Big( \sigma_0(q)\big( \| f\|_\infty +T_a^b f \big) \Big) .
\end{equation}
\begin{equation}\label{A2}
\sum\limits_{\substack{1\leqslant k\leqslant N \\ (k,q)=1}} \frac{\varphi(k)}{k} f(k) =C(\ell) \int_0^N f(x)\, dx +
O_\ell \Big( \big( \| f\|_\infty +T_0^N f \big) \log N \Big) .
\end{equation}
\begin{equation}\label{A3}
\sum\limits_{1\leqslant k\leqslant N} \frac{\varphi (\ell k)}{k} f(k) =\ell C(\ell) \int_0^N f(x)\, dx +
O_{\ell,\delta} \Big( \big( \| f\|_\infty +T_0^N f \big) N^\delta \Big) .
\end{equation}
\begin{equation}\label{A4}
\begin{split}
\sum\limits_{\substack{a\in I,b\in J \\ ab\equiv h \hspace{-6pt}\pmod{q} \\ (b,q)=1}} \hspace{-10pt} f(a,b) = &
\frac{\varphi(q)}{q^2} \iint_{I\times J} f(x,y)\, dxdy +O_\delta \big( T^2 \| f\|_\infty q^{1/2+\delta} (h,q)^{1/2}\big) \\
& + O_\delta \bigg( T \| \nabla f\|_\infty q^{3/2+\delta} (h,q)^{1/2} +
\frac{1}{T}\| \nabla f\|_\infty \vert I\vert \cdot \vert J\vert \bigg) .
\end{split}
\end{equation}
Proofs can be found for instance in \cite[Lemma 2.2]{BCZ0}, \cite[Lemmas 2.1 and 2.2]{BG},
and respectively in \cite[Proposition A4]{BZ}.

\subsection*{Acknowledgements}
The second author acknowledges support from Department of Education Grant
P200A090062,``University of Illinois GAANN Mathematics Fellowship Project."
The third author acknowledges support from National Science Foundation grant
DMS 08-38434 ``EMSW21-MCTP: Research Experience for Graduate Students."
We are grateful to the referee for careful reading and pertinent remarks.


\begin{thebibliography}{99}

\bibitem{ABCZ} V. Augustin, F.\,P. Boca, C. Cobeli, A. Zaharescu, \textit{The $h$-spacing
distribution between Farey points}, Math. Proc. Cambridge Philos. Soc. \textbf{131} (2001), 23--38.

\bibitem{BH} D. A. Badziahin, A. K. Haynes, \textit{A note on Farey fractions with denominators in arithmetic progressions}, Acta Arith. \textbf{147} (2011), 205--215.

\bibitem{BCZ0} F.~P.~Boca, C.~Cobeli, and A.~Zaharescu, \textit{Distribution of lattice points visible from the origin}, Comm.~Math.~Phys.~213 (2000), 433-470.

\bibitem{BCZ} F. P. Boca, C. Cobeli, A. Zaharescu, \textit{A conjecture of R. R. Hall on Farey arcs}, J. Reine Angew. Math. \textbf{535} (2001), 207--236.

\bibitem{BG} F. P. Boca, R. N. Gologan, \textit{On the distribution of the free path length of the linear flow in a honeycomb},
Ann. Inst. Fourier (Grenoble) {\bf 59} (2009), 1043--1075.

\bibitem{BGZ} F. P. Boca, R. N. Gologan, A. Zaharescu, \textit{On the index of Farey sequences}, Q. J. Math. \textbf{53} (2002), 377--391.

\bibitem{BZ} F. P. Boca, A. Zaharescu, \textit{The correlations of Farey fractions}, J. Lond. Math. Soc. (2) {\bf 72} (2005), 25--39.

\bibitem{BZ1} F. P. Boca, A. Zaharescu, \textit{On the correlation of directions in the Euclidean space},
Trans. Amer. Math. Soc. {\bf 358} (2006), 1797--1825.

\bibitem{Hall} R. R. Hall, \textit{A note on Farey series}, J. Lond. Math. Soc. (2) \textbf{2} (1970), 139--148.

\bibitem{H} A. K. Haynes, \textit{Numerators of differences of nonconsecutive Farey fractions}, Int. J. Number Theory \textbf{6} (2010), 655--666.

\bibitem{XZ} M. Xiong, A. Zaharescu, \textit{Correlation of fractions with divisibility constraints}, Math. Nachr. \textbf{284}
(2011), 393--407.

\end{thebibliography}
\end{document}